\documentclass[11pt]{amsart}
\usepackage{amsmath,amssymb,txfonts,mathtools}
\usepackage{enumitem,microtype,needspace}
\usepackage[colorlinks=true,linkcolor=blue,citecolor=blue,urlcolor=blue]{hyperref}
\usepackage[nameinlink,capitalise,noabbrev]{cleveref}
\allowdisplaybreaks[1]
\numberwithin{equation}{section}
\theoremstyle{plain}
\newtheorem{theorem}{Theorem}[section]
\newtheorem{proposition}[theorem]{Proposition}
\newtheorem{lemma}[theorem]{Lemma}
\newtheorem{corollary}[theorem]{Corollary}

\newtheorem{remark}[theorem]{Remark}
\newcommand{\Ric}{\operatorname{Ric}}
\newcommand{\Sc}{\operatorname{R}}

\newcommand{\diam}{\operatorname{diam}}
\newcommand{\Area}{\operatorname{Area}}
\newcommand{\Vol}{\operatorname{Vol}}
\newcommand{\dd}{\,\mathrm d}
\newcommand{\Sph}{\mathbb S}
\newcommand{\R}{\mathbb R}

\newcommand{\Hess}{\nabla^2}
\newcommand{\II}{\mathrm{II}}
\crefname{theorem}{Theorem}{Theorems}
\crefname{proposition}{Proposition}{Propositions}
\crefname{lemma}{Lemma}{Lemmas}
\crefname{corollary}{Corollary}{Corollaries}
\crefname{definition}{Definition}{Definitions}

\hypersetup{pdftitle={Integral of Scalar Curvature under Volume Collapse},pdfauthor={Guoyi Xu}}
\begin{document}

\title[Integrals of scalar curvature]{Integral of Scalar Curvature under Volume Collapse}
\author{Guoyi Xu}
\address{Guoyi Xu\\ Department of Mathematical Sciences\\Tsinghua University, Beijing\\P. R. China}
\email{guoyixu@tsinghua.edu.cn}
\date{September 7, 2026}

\begin{abstract}
We construct a smooth complete metric on $\R^3$ with strictly positive
Ricci curvature and $\displaystyle \lim_{s\rightarrow\infty}s^{-1}\int_{B_p(s)}\Sc\dd\mu= \infty$. This gives a negative answer to Shing-Tung Yau's
question on the asymptotic scalar-curvature integral. We also give  counterexamples to Naber's local Shing-Tung Yau conjecture on
compact manifolds. The noncompact example has zero asymptotic volume ratio, and the compact counterexamples have volumes tending to zero.
\end{abstract}

\thanks{Guoyi Xu was partially supported by NSFC 12141103.}
\maketitle
\tableofcontents

\section{Introduction}
\label{sec:introduction}

We consider the boundedness question for the scalar-curvature case of
Shing-Tung Yau's problem \cite[Problem~9]{Yau1992}, as formulated in
\cite[Question~1.1]{Yang2013}; see also the finite-limit formulation
\cite[Question~1.3]{Xu2024}. The boundedness assertion is that every fixed complete
noncompact $(M^n,g)$ with $\Ric_g\geq0$ satisfies, at every $p\in M$,
\begin{align}
 \varlimsup_{r\to\infty} r^{2-n}
       \int_{B_p^g(r)}\Sc_g\dd\mu_g<\infty.
 \label{eq:asym-three-unrestricted-question}
\end{align}

The local assertion is Naber's conjecture \cite[Conjecture~3.5]{Naber2020} as follows:
\begin{align}
 \Ric_g(M^n)\geq-\Lambda g
 \quad\Longrightarrow\quad
 \int_{B_p^g(1)}\Sc_g\dd\mu_g\leq C(n,\Lambda),
 \qquad \Lambda\geq0,
 \label{eq:intro-Naber-local}
\end{align}

For a complete noncompact three-manifold with $\Ric_g\geq0$, define
\begin{align}
 \operatorname{AVR}(g):=
 \lim_{r\to\infty}\frac{\Vol_g B_p^g(r)}{(4\pi/3)r^3}.
 \label{eq:intro-AVR}
\end{align}
 A point $o$ is a pole if
$\exp_o^g:T_oM\to M$ is a diffeomorphism. Under the pole assumption, Bo Zhu
\cite[Theorem~1.7]{Zhu2022geometry} proved that the limit superior
in \eqref{eq:asym-three-unrestricted-question} is finite in dimension
three.
Guoyi Xu \cite[Theorem~1.4]{Xu2024} subsequently obtained the exact formula
\begin{align}
 \lim_{r\to\infty}\frac1r
       \int_{B_p^g(r)}\Sc_g\dd\mu_g
 =8\pi\bigl(1-\operatorname{AVR}(g)\bigr),
 \qquad p\in M,
 \label{eq:intro-pole-integral}
\end{align}
for every complete three-manifold with $\Ric_g\geq0$ admitting a pole.

There are also estimates with a Green-function weight. Let $(M^3,g)$ be complete and noncompact with $\Ric_g\geq0$.
Suppose that the minimal positive Green function $G_p$ exists, with
$-\Delta_gG_p=\delta_p$, where
$\Delta_g=\operatorname{div}_g\nabla$ and $\delta_p$ is the unit mass
at $p$; this is the non-parabolic case. Put
$b_p=(4\pi G_p)^{-1}$ on $M\setminus\{p\}$ and $b_p(p)=0$.
Guoyi Xu \cite[Theorem~1.2]{Xu2020} proved
\begin{align}
 \limsup_{t\to\infty}\frac1t
       \int_{\{b_p\leq t\}}\Sc_g|\nabla b_p|_g\dd\mu_g
 \leq8\pi\bigl(1-\operatorname{AVR}(g)\bigr).
 \label{eq:intro-Green-integral}
\end{align}
Related Green-function comparison and scalar-curvature integral
estimates were developed by Bo Zhu \cite{Zhu2022comparison}.
Zixuan Chen, Guoyi Xu and Shuai Zhang \cite[Theorem~1.5]{ChenXuZhang2026} obtained
existence of the limit and equality in \eqref{eq:intro-Green-integral} when
$\operatorname{AVR}(g)>0$. 

For a complete connected noncompact three-manifold with
$\Ric_g\geq0$, Munteanu and Jiaping Wang
\cite[Theorem~1.2]{MunteanuWang2025}
proved
\begin{align}
 \varlimsup_{r\to\infty}\frac1r
       \int_{B_p^g(r)}\Sc_g\dd\mu_g\leq8\pi
 \label{eq:intro-Munteanu-Wang}
\end{align}
under the assumptions $0<c_0\leq\Sc_g\leq C_0<\infty$
and that $M^3$ has exactly one end.

Shiguang Ma \cite[Theorem~1.3]{Ma2026} obtained explicit asymptotic
scalar-curvature integral formulas for complete noncompact locally
conformally flat manifolds of dimension $n\geq3$ with
$\Ric_g\geq0$; in particular, the limit in
\eqref{eq:asym-three-unrestricted-question} is finite in this class.
In dimension three, Jialong Deng \cite[Theorem~B(i)]{Deng2026} proved that
if $(M^3,g)$ is complete, connected, noncompact, oriented and nonflat,
$\Ric_g\geq0$, and there are $C_0,r_0>0$ such that
\begin{align}
 \Sc_g(x)\leq C_0d_g(p,x)^{-2}
 \quad\text{whenever }d_g(p,x)\geq r_0,
 \label{eq:intro-quadratic-decay}
\end{align}
then the limit in
\eqref{eq:asym-three-unrestricted-question} is zero when
$\operatorname{AVR}(g)=0$, and is at most
$8\pi(1-\operatorname{AVR}(g))$ when $\operatorname{AVR}(g)>0$.

Related scalar-curvature integral results are available in the
K\"ahler setting.
Gang Liu \cite[Theorem~2]{Liu2024} proved that, for every complete
noncompact K\"ahler manifold with nonnegative holomorphic
bisectional curvature,
\begin{align}
 \lim_{r\to\infty}\frac{r^2}{\Vol_gB_p^g(r)}
       \int_{B_p^g(r)}\Sc_g\dd\mu_g
 \label{eq:intro-Liu-limit}
\end{align}
exists in $[0,+\infty]$.

Gang Liu and Sz\'ekelyhidi \cite[Proposition~1.7]{LiuSzekelyhidi2022} proved a local bound on polarized K\"ahler manifolds of complex
dimension $m$: if $\Ric_g>-g$ and
$\Vol_gB_p^g(1)>v>0$, then
\begin{align}
 \int_{B_p^g(1)}\Sc_g\dd\mu_g\leq C(m,v).
 \label{eq:intro-Liu-Szekelyhidi}
\end{align}

The following estimate is due to Wenshuai Jiang and Naber
, which was announced in
\cite[Theorem~2.17]{Naber2020} and restated in  
\cite[Theorem~1.7]{CucinottaMondino2026} by Cucinotta and Mondino, whose Appendix sketches
a proof: for $n\geq2$, $\Lambda\geq0$, $v>0$
and $s\in(0,1)$, every complete $(M^n,g,p)$ with
$\Ric_g\geq-\Lambda g$ and $\Vol_gB_p^g(1)\geq v$ satisfies
\begin{align}
 \frac{1}{\Vol_g B_p^g(1)}
       \int_{B_p^g(1)}|\Ric_g|_g^s\dd\mu_g
 \leq C(n,\Lambda,v,s).
 \label{eq:intro-Jiang-Naber}
\end{align}

The main result of this paper is the following result.
\begin{theorem}\label{thm:main}
There exists a smooth complete metric $g$ on $\R^3$ with
$\Ric_g>0$ such that
\begin{align}
 \lim_{s\to\infty}\frac1s\int_{B_p^g(s)}\Sc_g\dd\mu_g=+\infty
 \qquad \forall p\in\R^3.
 \nonumber 
\end{align}
There exists a sequence of smooth metrics $g_j^{\mathrm c}$ on
$\Sph^3$ such that
\begin{align}
 \Ric_{g_j^{\mathrm c}}>0,\qquad
 \diam(\Sph^3,g_j^{\mathrm c})<\tfrac12,\qquad
 \lim_{j\to\infty}\int_{\Sph^3}
       \Sc_{g_j^{\mathrm c}}\dd\mu_{g_j^{\mathrm c}}=+\infty.
 \nonumber 
\end{align}
\end{theorem}

Theorem~\ref{thm:main} gives a negative answer to Shing-Tung Yau's question, and disproves Naber's conjecture for complete non-compact and compact ambient manifolds. 

Note our examples do not conflict the assertion
\cite[Conjecture~2.18]{Naber2020}, which assumes a fixed
positive lower bound for the unit-ball volume.

Our construction modifies the metric form
$g=dr^2+a(r,t)^2dt^2+b(r,t)^2d\phi^2$
used in Perelman's construction on a finite cylinder
\cite[Section~2, p.~159]{Perelman1997Betti}
by allowing the coefficient of $dr^2$ to depend on both $r$ and $t$.
We construct
\begin{align}
 g_E
 &=N(r,t)^2dr^2+a(r,t)^2dt^2+b(r,t)^2d\phi^2,
 \qquad E:=[1,\infty)\times\mathbb{S}^2,
 \label{eq:general-cylinder}
\end{align}
with strictly positive Ricci curvature.
After matching the boundary metrics and verifying the required
inequality for the second fundamental forms, we apply Perelman's
gluing theorem to attach a closed three-ball with strictly positive
Ricci curvature along $\partial E$.
The resulting smooth complete metric $g$ on $\mathbb{R}^3$ satisfies
\begin{align*}
 \lim_{s\to\infty}\frac1s
 \int_{B_p^g(s)}\Sc_g\,d\mu_g=+\infty
 \qquad\text{for every }p\in\mathbb{R}^3.
\end{align*}

Instead of giving the explicit metric $g_E$ at the beginning and just verifying the conclusion afterwards, we derives the choices of $b$, $a$, and $N$ successively, using the prescribed symmetries, the smoothness
conditions at the poles and equator, and explicit differential
inequalities ensuring $\Ric_{g_E}>0$.
We include these calculations to explain the choices of the
coefficients and the restrictions on their parameters (see Section~\ref{sec:log-construction} for details). We believe in this procedure has its own interest.

After completing this work, we became aware of the preprint
of Tianze Hao and Jintian Zhu~\cite{HaoZhu2026}, which also gives a counterexample
to Shing-Tung Yau's conjectured asymptotic scalar-curvature integral bound.
The present work was carried out independently of theirs.

\section{Construction of the outer cylinder}
\label{sec:log-construction}

The symbol $\Sph^k$ denotes the unit sphere in
$\R^{k+1}$; its round metric is denoted by $g_{\Sph^k}$.

The notation $\Vol_g(E)$ means $\int_E\dd\mu_g$, and
$\Area_g(\Sigma)$ means the area in the induced surface metric.
For a surface metric $h$, $K_h$ denotes its Gaussian curvature.

\textbf{(1). The rotationally symmetric case}.
\begin{lemma}\label{lem:rotational-obstruction}
Suppose $\dd s^2+f(s)^2g_{\Sph^2}$ is a smooth complete metric
on $\R^3$ with a pole $o$ at $s=0$, where $f>0$ on $(0,\infty)$,
$f(0)=0$, and $f'(0)=1$. If its Ricci tensor is nonnegative, then
\begin{align}
 \int_{B_o(S)}\Sc\dd\mu\leq16\pi S\qquad(S>0).
 \label{eq:rotational-obstruction}
\end{align}
\end{lemma}
\begin{proof}
The radial Ricci component is $-2f''/f$, so $f''\leq0$ and
$f'\leq1$. A negative value of $f'$ would imply
$f(s)\leq f(s_0)+(s-s_0)f'(s_0)$ for $s\geq s_0$, contradicting
$f>0$. Thus $0\leq f'\leq1$. 

The scalar curvature formula and integration by parts give
\begin{align}
 \int_{B_o(S)}\Sc\dd\mu
 &=4\pi\int_0^S\{-4ff''+2(1-(f')^2)\}\dd s\nonumber\\
 &=-16\pi f(S)f'(S)+8\pi S+8\pi\int_0^S(f')^2\dd s
 \leq16\pi S.
\end{align}
The zero-end boundary term vanishes because $f(0)=0$.
\end{proof}

\textbf{(2). Choice of the cross-sectional metric}.  The preceding lemma shows that $g=dr^2+F(r)^2g_{\Sph^2}$ is too restrictive. For related constructions with nonnegative Ricci curvature in which
the cross-sectional metric varies with the radial parameter, see
Colding and Naber~\cite[Lemma~2.1 and its proof]{ColdingNaber2013}.

We seek smooth cross-sectional metrics
\begin{align*}
 h_r=\dd s^2+B_r(s)^2\dd\phi^2,
 \qquad 0\leq s\leq2\ell_r,
 \qquad \phi\in\R/(2\pi\mathbb Z),
\end{align*}
where $s$ is arclength from the north pole, $s=\ell_r$ is the
equator, and $B_r(2\ell_r-s)=B_r(s)$. Requiring $K_r>0$ gives
$B_r''=-K_rB_r<0$ between the poles. Since $B_r'(\ell_r)=0$,
$B_r$ is strictly monotone on each open hemisphere. Thus
\begin{align*}
 F(r):=B_r(\ell_r),\qquad t:=\frac{B_r(s)}{F(r)}
\end{align*}
defines a coordinate with pole value $0$ and equatorial value $1$.
In this coordinate,
\begin{align*}
 h_r=a(r,t)^2\dd t^2+F(r)^2t^2\dd\phi^2.
\end{align*}
This latitude coordinate is inspired by
\cite[Section~2]{Perelman1997Betti}.
We now prescribe $g=\dd r^2+h_r$ in the coordinates $(r,t,\phi)$.
For $F(r)=r^\alpha$, where $r\geq1$ and $0<\alpha<1$,
\begin{align*}
 \Ric_g(\partial_r,\partial_r)
 =-\frac{a_{rr}}a+\frac{\alpha(1-\alpha)}{r^2}.
\end{align*}

\textbf{Step (3). Definition of \(q\) and the pole condition}
Define
\begin{align}
 q(r,t):=\frac{r^2}{a(r,t)^2}.
 \label{eq:q-from-a}
\end{align}
Then
\begin{align}
 K_r=-\frac{q_t}{2r^2t}.
 \label{eq:Kr-from-q}
\end{align}
Smoothness at a pole requires $a(r,0)=r^\alpha$. Therefore, with
\begin{align}
 \lambda:=2-2\alpha,\qquad q(r,0)=r^\lambda,
 \label{eq:lambda-alpha}
\end{align}
the pole normalization is satisfied.

\textbf{Step (4). Introduction of $N$ and the weighted curvature integral.}
Write $\Sigma_r=\{r\}\times\Sph^2$, let $\dd A_r$ be its area
measure, and put $b=r^\alpha t$. We introduce
\begin{align}
 g=N(r,t)^2\dd r^2+h_r,
 \qquad N>0.
 \label{eq:introduce-N}
\end{align}
Since $\dd\mu_g=N\dd r\dd A_r$, the relevant cross-sectional
integral is $\int_{\Sigma_r}NK_r\dd A_r$, whereas
\begin{align}
 \int_{\Sigma_r}K_r\dd A_r=4\pi.
 \label{eq:why-N-Gauss-Bonnet}
\end{align}

\begin{remark}\label{remark:why-N-must-depends-on-r-and-t}
If $N=N(r)$, set $\widetilde r=\int_1^rN(\rho)\dd\rho$. Then
\begin{align*}
 \left(\int_{\Sigma_r}NK_r\dd A_r\right)\dd r
 =4\pi\dd\widetilde r.
\end{align*}
If $N=N(t)$, positivity of $K_r$ gives
\begin{align}
 \int_{\Sigma_r}N(t)K_r\dd A_r
 \leq4\pi\max_{0\leq t\leq1}N(t).
 \label{eq:N-t-only-no-growth}
\end{align}
These observations concern only this weighted integral; they do
not exclude other constructions with $N=N(r)$ or $N=N(t)$.
\end{remark}

We take $N=N(q(r,t))$ and impose $q(r,1)=0$, so that
$N(r,1)=N(0)$ is independent of $r$.

\textbf{Step (5). Sign conditions on \(q\) and $N$}
We require $q_t<0$. On either hemisphere,
$\dd A_r=r^{1+\alpha}tq^{-1/2}\dd t\dd\phi$.
Changing variables from $t$ to $q$ and adding the two hemispheres gives
\begin{align}
 \int_{\Sigma_r}NK_r\dd A_r
 =2\pi r^{-\lambda/2}
   \int_0^{r^\lambda}\frac{N(u)}{\sqrt u}\dd u.
 \label{eq:N-weighted-K-lambda}
\end{align}
If $N_q\leq0$, this integral is at most $4\pi N(0)$.
We therefore seek $N_q>0$.

\textbf{Step (5). Choice of $q$.}
Proposition~\ref{prop:two-variable-Ricci} gives
\begin{align}
 r^2N^2R_{00}
 &=-r^2\frac{a_{rr}}a+\alpha(1-\alpha)
 +\left(r\frac{a_r}{a}+\alpha\right)\frac{rN_r}{N}
 \nonumber\\
 &\quad
 -N\left\{
 qN_{tt}+\left(\frac{q_t}{2}+\frac qt\right)N_t
 \right\}.
 \label{eq:R00-complete-general}
\end{align}
Since $N=N(q)$,
\begin{align}
-\frac{qN_{tt}+(q_t/2+q/t)N_t}{Nr^2}=-\frac{qN_q}{Nr^2}\left(q_{tt}+\frac{q_t}t\right)-\frac{qN_{qq}+N_q/2}{Nr^2}q_t^2.
\end{align}
Thus $q_{tt}+q_t/t<0$ makes the first term on the right positive. The remaining terms and $R_{01}$ will be estimated below.

The equation $q_{tt}+q_t/t=0$, with $q_t<0$ and $q(1)=0$,
has solutions $q(t)=-c\log t$, $c>0$.
Taking $c=2$ gives the model $-\log t^2$.
To replace its infinite pole value by $r^\lambda$ while retaining
$q(r,1)=0$, prescribe the affine function of $t^2$
\begin{align}
 e^{-q(r,t)}=t^2+(1-t^2)e^{-r^\lambda}.
 \label{eq:q-exp-general}
\end{align}
Thus
\begin{align}
 q(r,t)=-\log\bigl(t^2+(1-t^2)e^{-r^\lambda}\bigr).
 \label{eq:q-general-form}
\end{align}

Direct differentiation gives
\begin{align}
 q_t
 &=
 -2t(1-e^{-r^\lambda})e^q<0,
 \label{eq:q-general-qt}\\
 q_{tt}+\frac{q_t}{t}
 &=
 -4(1-e^{-r^\lambda})
 e^{\,2q-r^\lambda}<0.
 \label{eq:q-general-qtt}
\end{align}
Therefore
\begin{align}
 r^2K_r
 =
 (1-e^{-r^\lambda})e^q>0.
 \label{eq:q-general-K}
\end{align}

The pole smoothness follows from
\begin{align}
 q(r,t)
 =
 r^\lambda
 -
 \log\left(
 1+(e^{r^\lambda}-1)t^2
 \right),
 \label{eq:q-pole-smooth}
\end{align}
and, writing
\[
v^2=1-t^2
\]
near the equator,
\begin{align}
 q(r,\sqrt{1-v^2})
 =
 -\log\left(
 1-(1-e^{-r^\lambda})v^2
 \right),
 \label{eq:q-equator-smooth}
\end{align}
so that
\[
\frac{q(r,\sqrt{1-v^2})}{v^2}
\longrightarrow
1-e^{-r^\lambda}>0.
\]

\textbf{Step (6). Choice of $N$.}
For $A,B>0$ and a smooth positive function
$F:[0,\infty)\to(0,\infty)$, set $N^2=A+BF(q)$.
Equations~\eqref{eq:q-general-qt}--\eqref{eq:q-general-K} give
\begin{align}
 &-\frac{qN_{tt}+(q_t/2+q/t)N_t}{Nr^2}
 \nonumber\\
 &\quad=\frac{BK_r}{N^2}
 \left\{2qe^{q-r^\lambda}F'(q)
 -(1-e^{q-r^\lambda})\bigl(F'(q)+2qF''(q)\bigr)\right\}
 \nonumber\\
 &\qquad\quad
 +\frac{B^2qK_r(1-e^{q-r^\lambda})F'(q)^2}{N^4}.
 \label{eq:N-choice-direct}
\end{align}
Thus $F'(q)>0$ and $F'(q)+2qF''(q)<0$ for all sufficiently
large $q$ make the right-hand side nonnegative there.
The equality equation has solutions $c_0+c_1\sqrt q$ on $q>0$,
where $c_0,c_1\in\R$. A nonzero $\sqrt q$ term is not smooth
at the equator, since
\begin{align*}
 \sqrt{q(r,\sqrt{1-v^2})}
 =\sqrt{1-e^{-r^\lambda}}\,|v|+O(|v|^3)
 \qquad(v\to0).
\end{align*}
For the smooth choice $F(q)=(1+q)^\beta$, with $\beta>0$,
\begin{align}
 F'(q)+2qF''(q)
 =\beta(1+q)^{\beta-2}\bigl(1+(2\beta-1)q\bigr).
 \label{eq:N-choice-beta}
\end{align}
The required eventual negative sign holds precisely for
$0<\beta<1/2$. We take $\beta=1/4$, giving
\begin{align}
 F'(q)+2qF''(q)=\frac{2-q}{8(1+q)^{7/4}}.
 \label{eq:N-choice-F-sign}
\end{align}
Consequently, the right-hand side of \eqref{eq:N-choice-direct}
is nonnegative for $q\geq2$ and is bounded below by
$-BK_r/(4N^2)$ for $0\leq q\leq2$.
We choose $A=16$ and $B=1/8$, so that
\begin{align}
 N^2&=16+\frac18(1+q)^{1/4},
 \label{eq:N-choice-final-form}\\
 N_q&=\frac1{64N(1+q)^{3/4}}.
 \label{eq:N-choice-Nq}
\end{align}
Here $N>4$, $N_t=N_q q_t$, and $N_r=N_q q_r$.
The following estimates verify that these choices suffice.

\textbf{Step (7). Choice of \(\lambda\).}
Throughout this step, \(r\geq1\), \(0<t<1\), and \(0<\lambda\leq1\).
Since
\begin{align*}
 \frac12
 \leq \frac{d}{dq}\log\frac{e^q-1}{q}
 =\frac1{1-e^{-q}}-\frac1q
 \leq1,
\end{align*}
direct differentiation gives
\begin{align}
 e^{q-r^\lambda}
 &\leq
 \frac{rq_r}{\lambda q}
 =\frac{r^\lambda(e^q-1)}{q(e^{r^\lambda}-1)}
 \leq e^{(q-r^\lambda)/2}\leq1,
 \label{eq:lambda-choice-xqx}\\
 \frac{1-e^{q-r^\lambda}}2
 &\leq1-\frac{rq_r}{\lambda q}
 \leq1-e^{q-r^\lambda},
 \qquad
 r^\lambda\frac{rq_r}{\lambda q}\leq\max\{2,q\}.
 \label{eq:lambda-choice-quotient-bounds}
\end{align}
Moreover,
\begin{align}
 q_{rr}
 &=q_r^2-\frac{1-\lambda+\lambda r^\lambda}{r}q_r,
 \nonumber\\
 \frac{r^2a_{rr}}a
 &=\frac{\lambda^2}{4}
 \left\{
 \left(2r^\lambda-2-\frac2\lambda\right)
 \frac{rq_r}{\lambda q}
 +(3-2q)\left(\frac{rq_r}{\lambda q}\right)^2
 \right\},
 \label{eq:lambda-choice-arr}\\
 -\frac{r^2b_{rr}}b
 &=\frac{\lambda(2-\lambda)}4,
 \qquad
 1-\frac{\lambda}{2}\leq\frac{ra_r}{a}\leq1.
 \label{eq:lambda-choice-brr}
\end{align}

Differentiating \eqref{eq:N-choice-Nq} and using
\eqref{eq:q-general-qt}--\eqref{eq:q-general-K}, we obtain
\begin{align*}
 N_{qq}
 &=-N_q\left(\frac{N_q}{N}+\frac{3}{4(1+q)}\right)<0,\\
 q q_{tt}+\frac{q_t^2}{2}
 &=2r^2K_r\left\{1+q-(1+2q)e^{q-r^\lambda}\right\}
 \leq2r^2K_r(1+q).
\end{align*}

Consequently,
\begin{align*}
 \frac{r^2a_{rr}}a
 &\leq\frac{\lambda^2}{2}\max\{2,q\}
 \leq2\lambda^2r^2K_r,\\
 \frac{qN_{tt}+q_tN_t/2}{r^2N}
 &\leq\left(\frac14-\frac4{N^2}\right)K_r,\\
 R_{11}
 &\geq
 \left(\frac34+\frac{2+\lambda-2\lambda^2}{N^2}\right)K_r
 \geq\frac34K_r,\\
 R_{22}
 &\geq\left(1-\frac{(2-\lambda)^2}{2N^2}\right)K_r>0.
\end{align*}

The mixed component and the derivative of \(N\) satisfy
\begin{align}
 R_{01}
 &=\frac{\lambda\sqrt q}{2Nr^2t}
 \left\{
 1-\frac{rq_r}{\lambda q}
 +\frac{2-\lambda}{\lambda}\frac{tN_t}{N}
 \right\},
 \label{eq:lambda-choice-R01}\\
 0\leq-\frac{tN_t}{N}
 &=\frac{1-e^{q-r^\lambda}}{32N^2(1+q)^{3/4}}
 \leq\frac{1-e^{q-r^\lambda}}{512}.
 \label{eq:radial-bound-use-lapse-control}
\end{align}
Thus, for \(\lambda\geq2/257\),
\begin{align*}
 0
 &\leq
 1-\frac{rq_r}{\lambda q}
 +\frac{2-\lambda}{\lambda}\frac{tN_t}{N}
 \leq1-\frac{rq_r}{\lambda q},\\
 \frac{4r^2N^2}{\lambda^2}\frac{R_{01}^2}{R_{11}}
 &\leq\frac{4q}{3}\left(1-\frac{rq_r}{\lambda q}\right).
\end{align*}
By \eqref{eq:N-choice-direct} and \eqref{eq:N-choice-F-sign},
\begin{align*}
 -N\left\{
 qN_{tt}+\left(\frac{q_t}{2}+\frac qt\right)N_t
 \right\}
 &\geq\frac{r^2K_r(q-2)}{64(1+q)^{7/4}}\\
 &\geq
 \begin{cases}
 -\dfrac1{32},&0<q\leq2,\\[3pt]
 \dfrac{(e^q-1)(q-2)}{64(1+q)^2},&q\geq2.
 \end{cases}
\end{align*}

For the bound on $0<q\leq2$, use $r^2K_r\leq e^q$ and
\begin{align*}
 \frac{d}{dq}\log\frac{e^q(2-q)}{(1+q)^{7/4}}
 =\frac{-4q^2+7q-10}{4(2-q)(1+q)}<0
 \qquad(0<q<2).
\end{align*}

For \(0<q\leq5\), \eqref{eq:lambda-choice-xqx} gives
\begin{align*}
 &\left(2+\frac2\lambda-2r^\lambda+\frac{4q}{3}\right)
 \frac{rq_r}{\lambda q}
 +(2q-3)\left(\frac{rq_r}{\lambda q}\right)^2\\
 &\qquad\geq
 \left(\frac2\lambda-\frac43-2(r^\lambda-q)\right)
 \frac{rq_r}{\lambda q}
 \geq-4e^{-1/(2\lambda)-2/3}.
\end{align*}

Substitution into \eqref{eq:R00-complete-general} therefore yields,
for \(2/257\leq\lambda\leq1\),
\begin{align}
 &\frac{4r^2N^2}{\lambda^2}
 \left(R_{00}-\frac{R_{01}^2}{R_{11}}\right)
 \nonumber\\
 &\quad\geq
 \begin{cases}
 \displaystyle
 \frac2\lambda-\frac{11}{3}-\frac1{8\lambda^2}
 -4e^{-1/(2\lambda)-2/3},
 &0<q\leq2,\\[5pt]
 \displaystyle
 \frac2\lambda-\frac{23}{3}
 -4e^{-1/(2\lambda)-2/3},
 &2\leq q\leq5,\\[5pt]
 \displaystyle
 \frac2\lambda-1-\frac{10q}{3}
 +\frac{(e^q-1)(q-2)}{16\lambda^2(1+q)^2},
 &q\geq5.
 \end{cases}
 \label{eq:lambda-choice-determinant}
\end{align}

The sufficient inequalities $1/(16\lambda^2)\geq1$ and
$2/\lambda-1/(8\lambda^2)\geq6$ are equivalent, for
$\lambda>0$, to
\begin{align}
 \frac1{12}\leq\lambda\leq\frac14.
 \label{eq:lambda-choice-interval}
\end{align}
On this interval, $4e^{-1/(2\lambda)-2/3}\leq4e^{-8/3}<1/3$.
For $q\geq5$,
\begin{align*}
 \frac{e^5-1}{12}>11,\qquad
 \frac{d}{dq}\log\frac{(e^q-1)(q-2)}{(1+q)^2}>\frac23,
\end{align*}
and therefore
\begin{align*}
 7-\frac{10q}{3}+\frac{(e^q-1)(q-2)}{(1+q)^2}
 >\frac43+4(q-5).
\end{align*}
It follows that every lower bound in
\eqref{eq:lambda-choice-determinant} is at least
$1/3-4e^{-8/3}>0$.

By \eqref{eq:N-weighted-K-lambda} and $N\geq q^{1/8}/\sqrt8$,
\begin{align}
 \int_{\Sigma_r}NK_r\dd A_r
 \geq\frac{4\sqrt2\,\pi}{5}\,r^{\lambda/8}.
 \label{eq:lambda-choice-growth}
\end{align}
The exponent increases with $\lambda$. We choose the largest
value in \eqref{eq:lambda-choice-interval}:
\begin{align}
 \lambda=\frac14,\qquad\alpha=\frac78.
 \label{eq:q-lambda-final}
\end{align}

\begin{proposition}\label{prop:log-positive}
On $E=[1,\infty)\times\Sph^2$ define
\begin{align}
 q(r,t)&=-\log\bigl(t^2+(1-t^2)e^{-r^{1/4}}\bigr),\quad r\geq 1, t\in [0, 1], \nonumber\\
 h_r&=\frac{r^2}{q}\dd t^2+r^{7/4}t^2\dd\phi^2, \quad 0\le t\le1,\quad \phi\in\R/(2\pi\mathbb Z)\nonumber\\
 g_E&=\left(16+\frac{(1+q)^{1/4}}8\right)\dd r^2+h_r,
 \label{eq:log-metric}
\end{align}
then the metric $g_E$ has strictly positive Ricci curvature and
$\Sc_{g_E}\ge K_r$.
\end{proposition}

\begin{proof}
Use the previously defined $N,a,b$ with $\lambda=1/4$ and
$\alpha=7/8$, and the frame of
Proposition~\ref{prop:two-variable-Ricci}.
Write $\Sigma_r=\{r\}\times\Sph^2$, $K_r=K_{h_r}$, and let
$\dd A_r$ be the area measure of $h_r$.

By \eqref{eq:q-pole-smooth}, $q,N,a,b/t$ are smooth in $(r,t^2)$
at either pole, with $a(r,0)=(b/t)(r,0)=r^{7/8}$.
Equation~\eqref{eq:q-equator-smooth} shows that $q/(1-t^2)$
extends smoothly and positively at the equator.
Lemma~\ref{lem:identical-hemisphere-smoothness} therefore proves
smoothness and positive definiteness on $E$.

For $r\geq1$ and $0<t<1$, \eqref{eq:lambda-choice-brr} gives
\begin{align}
 \frac78\leq\frac{ra_r}{a}\leq1.
 \label{eq:direct-a-derivatives}
\end{align}
The estimates in Step~(7), evaluated at $\lambda=1/4$, give
\begin{align}
 R_{11}&\geq\frac34K_r,
 \label{eq:R11}\\
 R_{22}&\geq\left(1-\frac{49}{32N^2}\right)K_r
 \geq\frac14K_r,
 \label{eq:R22}
\end{align}
where $N^2\geq16$.
Moreover, \eqref{eq:lambda-choice-determinant} and the positivity
estimates following \eqref{eq:lambda-choice-interval} imply
\begin{align*}
 64r^2N^2\left(R_{00}-\frac{R_{01}^2}{R_{11}}\right)
 \geq\frac13-4e^{-8/3}>0.
\end{align*}
Since $R_{02}=R_{12}=0$ by
Proposition~\ref{prop:two-variable-Ricci}, these inequalities prove
$\Ric_{g_E}>0$ on $0<t<1$.

For each fixed $r$, $K_r$ and $N$ have finite positive limits at
the poles and equator. The preceding bounds therefore give a
strictly positive lower limit for
\begin{align*}
 \det(R_{ij})
 =R_{11}R_{22}\left(R_{00}-\frac{R_{01}^2}{R_{11}}\right).
\end{align*}
The Ricci tensor is nonnegative definite at the endpoints by
continuity, and its determinant is positive there. Hence it is
positive definite on all of $E$. Finally,
\begin{align*}
 \Sc_{g_E}=R_{00}+R_{11}+R_{22}\geq K_r.
\end{align*}
\end{proof}

\section{Compact extension and the global counterexample}\label{sec:compact-extension}

Throughout this section the second fundamental form of a boundary is
\begin{align}
 \mathrm{II}_\gamma(X,Y)=-\gamma(\nabla_X \vec{n},Y),
 \label{eq:boundary-sign-definition}
\end{align}
where $\vec{n}$ is the inward unit normal and $X,Y$ are tangent to the
boundary.  

\begin{lemma}\label{lem:boundary-conformal}
Let $\gamma$ be a smooth metric on a three-manifold with boundary,
let $\psi$ be a smooth function up to the boundary, let $\sigma\geq0$ be constant,
and put $\widetilde\gamma=e^{-2\sigma\psi}\gamma$. Then
\begin{align}
 \Ric_{\widetilde\gamma}
 =\Ric_\gamma+\sigma\nabla^2\psi
       +\sigma^2\dd\psi\otimes\dd\psi
       +(\sigma\Delta\psi-\sigma^2|\nabla\psi|_\gamma^2)\gamma.
 \label{eq:opening-conformal-Ricci}
\end{align}
If $h$ is the induced boundary metric and $\vec{n}$ is the $\gamma$-inward
unit normal, then
\begin{align}
 \widetilde h=e^{-2\sigma\psi}h,\qquad
 \widetilde n=e^{\sigma\psi}\vec{n},\qquad
 \mathrm{II}_{\widetilde\gamma}
 =e^{-\sigma\psi}\bigl(\mathrm{II}_\gamma+\sigma \vec{n}(\psi)h\bigr).
 \label{eq:boundary-conformal-II}
\end{align}
In particular, if $\psi=0$, $\vec{n}(\psi)=1$, and
$\mathrm{II}_\gamma=0$ on the boundary, the boundary metric is
unchanged and $\mathrm{II}_{\widetilde\gamma}=\sigma h$.
\end{lemma}

\begin{proof}
The two connections satisfy
\begin{align}
 \widetilde\nabla_XY-\nabla_XY
 =-\sigma\bigl(X(\psi)Y+Y(\psi)X
                 -\gamma(X,Y)\nabla\psi\bigr).
 \label{eq:boundary-connection-change}
\end{align}
Substitution into the coordinate formula for Ricci curvature
gives \eqref{eq:opening-conformal-Ricci}.
Since $\widetilde n=e^{\sigma\psi}\vec n$, for boundary-tangent $X$,
\begin{align*}
 \widetilde\nabla_X\widetilde n
 =e^{\sigma\psi}\bigl(\nabla_X\vec n
                    -\sigma\vec n(\psi)X\bigr).
\end{align*}
Contracting with $-\widetilde\gamma$ gives
\eqref{eq:boundary-conformal-II}.
\end{proof}

For $r\ge1$ and $0\le t\le1$, the inequality
$-\log(1-v)\ge v$ for $0\le v<1$ and convexity of $-\log$ give
\begin{align}
 \frac12(1-t^2)
 \le (1-e^{-r^{1/4}})(1-t^2)
 \le q(r,t)
 \le r^{1/4}(1-t^2).
 \label{eq:qbound}
\end{align}

\begin{lemma}\label{lem:fixed-filling}
Define
\begin{align}
 D:=\Bigl\{
 &\bigl(\sqrt{1-t^2}\cos\theta,
         \sqrt{1-t^2}\sin\theta,
         t\cos\phi,t\sin\phi\bigr):\nonumber\\
 &0\le t\le1,\quad 0\le\theta\le\pi,\quad
 \phi\in\R/(2\pi\mathbb Z)\Bigr\}\subset\Sph^3.
 \label{eq:filling-domain}
\end{align}
At $t=0$, all values of $\phi$ with $\theta$ fixed represent
one point; at $t=1$, all values of $\theta$ with $\phi$ fixed
represent one point. On $0<t<1$, define
\begin{align}
 \Gamma
 &=\frac{\dd t^2}{-\log\bigl(t^2+(1-t^2)e^{-1}\bigr)}
   +\frac{-\log\bigl(t^2+(1-t^2)e^{-1}\bigr)}
          {(1-e^{-1})^2}\dd\theta^2
   +t^2\dd\phi^2,\nonumber\\
 \Gamma_D
 &=\exp\!\left(
   -\frac{\sqrt{-\log\bigl(t^2+(1-t^2)e^{-1}\bigr)}
          \sin\theta}{2(1-e^{-1})}
   \right)\Gamma.
 \label{eq:filling-definitions}
\end{align}
Then $(D, \Gamma_D)$ is a smooth closed three-ball with 
\begin{align}
 \Ric_{\Gamma_D}>\frac5{288}\Gamma,
 \qquad
 \Gamma_D|_{T\partial D}=h_1,
 \qquad
 \II_{\partial D}=\frac14h_1.
 \label{eq:cap-conclusion}
\end{align}
\end{lemma}

\begin{proof}
Near $t=0$, use coordinates $(\theta,t\cos\phi,t\sin\phi)$.
The functions $q(1,t)$ and
\begin{align*}
 \frac{q(1,t)^{-1}-1}{t^2}
\end{align*}
are smooth functions of $t^2$, with the second expression
extended at zero. Together with
\begin{align*}
 \dd(t\cos\phi)^2+\dd(t\sin\phi)^2
 =\dd t^2+t^2\dd\phi^2,
\end{align*}
this proves smoothness and positivity of $\Gamma$ at $t=0$.

Near $t=1$, use coordinates
\begin{align*}
 \bigl(\sqrt{1-t^2}\cos\theta,
       \sqrt{1-t^2}\sin\theta,\phi\bigr),
 \qquad \sqrt{1-t^2}\sin\theta\ge0,
\end{align*}
with $\phi$ restricted to a circle coordinate chart. The identity
\begin{align*}
 \frac{q(1,t)}{1-t^2}
 =(1-e^{-1})\int_0^1
 \frac{\dd u}{1-(1-e^{-1})(1-t^2)u}
\end{align*}
shows that this quotient is smooth and positive as a function
of $1-t^2$, with value $1-e^{-1}$ at $t=1$.
Writing the two transverse metric terms as
\begin{align*}
 &\frac{1-t^2}{t^2q(1,t)}
       \bigl(\dd\sqrt{1-t^2}\bigr)^2
 +\frac{q(1,t)}{(1-e^{-1})^2(1-t^2)}
       (1-t^2)\dd\theta^2,
\end{align*}
the two coefficients are smooth functions of $1-t^2$ with
the same positive value $(1-e^{-1})^{-1}$ at $t=1$.
Their difference is divisible by $1-t^2$; conversion to the
displayed Cartesian coordinates therefore gives a smooth
positive definite tensor. The term $t^2\dd\phi^2$ is smooth
and positive there.
Moreover,
\begin{align*}
 \frac{\sqrt{q(1,t)}}{1-e^{-1}}\sin\theta
 =\frac1{1-e^{-1}}
   \sqrt{\frac{q(1,t)}{1-t^2}}
   \bigl(\sqrt{1-t^2}\sin\theta\bigr)
\end{align*}
is smooth at $t=1$, and is smooth at $t=0$ by the first
coordinate calculation. Thus $\Gamma_D$ is also smooth and
positive definite on $D$.

Set
\begin{align*}
 \psi:=\frac{\sqrt{q(1,t)}\sin\theta}{1-e^{-1}},
 \qquad \Gamma_D=e^{-\psi/2}\Gamma.
\end{align*}
For $0<t<1$, use the $\Gamma$-orthonormal frame
\begin{align*}
 e_t=\sqrt{q(1,t)}\,\partial_t,\qquad
 e_\theta=\frac{1-e^{-1}}{\sqrt{q(1,t)}}\partial_\theta,
 \qquad e_\phi=t^{-1}\partial_\phi.
\end{align*}
Both $\Ric_\Gamma$ and $\Hess_\Gamma\psi$ are diagonal in this
frame. Proposition~\ref{prop:two-variable-Ricci} gives
\begin{align}
 \Ric_\Gamma(e_i,e_i)
 &=\begin{cases}
 2(1-e^{-1})e^{2q(1,t)-1},&i=t,\theta,\\
 2(1-e^{-1})e^{q(1,t)},&i=\phi.
 \end{cases}
 \label{eq:cap-Ricci}
\end{align}
In particular, $\Ric_\Gamma\geq2(e-1)e^{-2}\Gamma$.
Direct differentiation gives
\begin{align*}
 \Hess_\Gamma\psi(e_t,e_t)
 &=\sin\theta\sqrt{q(1,t)}e^{q(1,t)}
       \bigl(1-2e^{q(1,t)-1}\bigr),\\
 \Hess_\Gamma\psi(e_\theta,e_\theta)
 &=\sin\theta\,
   \frac{e^{q(1,t)}-e^{2q(1,t)-1}-(1-e^{-1})}{\sqrt{q(1,t)}},\\
 \Hess_\Gamma\psi(e_\phi,e_\phi)
 &=-\sin\theta\sqrt{q(1,t)}e^{q(1,t)}.
\end{align*}
For $0\leq u\leq1$, concavity of $e^u-e^{2u-1}$ and elementary
maximization give
\begin{align*}
 e^u-e^{2u-1}-(1-e^{-1})&\geq-(1-e^{-1})u,\\
 \sqrt u\,e^{1-2u}&\leq\frac{\sqrt e}{2},\qquad
 \sqrt u\,e^{-u}\leq\frac1{\sqrt{2e}}.
\end{align*}
Using $0\leq q(1,t)\leq1$ and $0\leq\sin\theta\leq1$, we obtain
\begin{align}
 \frac{\Hess_\Gamma\psi(e_i,e_i)+\Delta_\Gamma\psi}
      {\Ric_\Gamma(e_i,e_i)}
 \geq
 \begin{cases}
 -\dfrac3{2(1-e^{-1})}-\dfrac{\sqrt e}{4},&i=t,\\[3pt]
 -\dfrac1{1-e^{-1}}-\dfrac{\sqrt e}{2},&i=\theta,\\[3pt]
 -\dfrac3{2(1-e^{-1})}-\dfrac1{2\sqrt{2e}},&i=\phi.
 \end{cases}
 \label{eq:cap-Hess-bound}
\end{align}
All three lower bounds exceed $-3$. Hence
\begin{align}
 \Hess_\Gamma\psi+(\Delta_\Gamma\psi)\Gamma
 \geq-3\Ric_\Gamma.
 \label{eq:cap-Hess-Ric}
\end{align}
By \eqref{eq:qbound} at $r=1$,
\begin{align}
 |\dd\psi|_\Gamma^2
 =t^2e^{2q(1,t)}\sin^2\theta+\cos^2\theta
 \leq\frac e2<\frac32.
 \label{eq:cap-gradient}
\end{align}
Applying \eqref{eq:opening-conformal-Ricci} with $\sigma=1/4$ gives
\begin{align*}
 \Ric_{\Gamma_D}
 \geq\left(\frac{e-1}{2e^2}-\frac e{32}\right)\Gamma
 >\frac5{288}\Gamma.
\end{align*}
These bounds extend to the endpoint sets by smoothness.

On the faces $\theta=0,\pi$, the inward $\Gamma$-unit normals are,
respectively, $\pm(1-e^{-1})q(1,t)^{-1/2}\partial_\theta$.
There $\II_\Gamma=0$, $\psi=0$, $\vec n(\psi)=1$, and the
induced metric is $h_1$. Lemma~\ref{lem:boundary-conformal}, followed
by smooth extension across $t=0,1$, yields
\begin{align*}
 \Gamma_D|_{T\partial D}=h_1,\qquad
 \II_{\partial D}=\frac14h_1.
\end{align*}
\end{proof}

\begin{proposition}\label{prop:fixed3-extension}
There is a smooth complete metric $g$ on $\R^3$ with $\Ric_g>0$
which equals \eqref{eq:log-metric} on $r\ge2$.
\end{proposition}
\begin{proof}
At $r=1$ the inward unit normal of $E$ is $N^{-1}\partial_r$.
The two boundary principal curvatures in
convention~\eqref{eq:boundary-sign-definition} are
$\left.-a_r/(Na)\right|_{r=1}$ and $-7/(8N)$.
By \eqref{eq:direct-a-derivatives},
$\left.a_r/a\right|_{r=1}\le1$, while
$N\ge\sqrt{129/8}>4$. Hence
\begin{align}
 \II_{\partial D}+\II_{\partial E}
 \ge\left(\frac14-\frac1{\sqrt{129/8}}\right)h_1>0.
 \label{eq:fixed-filling-boundary-sum}
\end{align}
The boundary metrics are both $h_1$, and the two Ricci tensors
are strictly positive. 

Apply Theorem~\ref{thm:unified-positive-Ricci-gluing} in
a neighborhood whose part in $E$ is contained in $1\le r<2$.

The resulting metric is smooth, has positive Ricci curvature,
and retains the displayed exterior. The explicit boundary
identification is $\partial D$ and $\partial E$. 
\end{proof}

Let $D_0= (\mathbb{R}^3, g)\backslash \{r>2\}$. 

For $A\subset M$ and $x,y\in A$, define
\begin{align*}
 d_A(x,y):=
 \inf_{\substack{
   \gamma:[0,1]\to A\ \text{piecewise smooth}\\
   \gamma(0)=x,\ \gamma(1)=y}}
 \int_0^1|\dot\gamma(u)|_g\,\dd u.
\end{align*}

\begin{lemma}\label{lem-domain-with-respet-to-ball}
Fix $p_0\in D_0$ and choose $C_0\ge1$ such that
\begin{align*}
 \sup_{z\in D_0}d_{D_0}(p_0,z)<C_0.
\end{align*}
Then for every $p\in M$ and
\begin{align*}
 s\ge\max\{128,4(C_0+d_g(p,p_0))\},
 \qquad T:=\frac{s}{32},
\end{align*}
one has
\begin{align*}
 \{T/2\le r\le T\}\subset B_p^g(s).
\end{align*}
\end{lemma}

\begin{proof}
Put $\Omega_R=D_0\cup([2,R]\times\Sph^2)$ for $R\geq2$, and
let $e_r=(r,(0,1,0))\in\Sigma_r$.
Two meridian arcs through a pole connect $e_r$ to any
$y\in\Sigma_r$. By \eqref{eq:qbound}, their total length gives
\begin{align*}
 d_{\Sigma_r}(e_r,y)
 \leq2r\int_0^1\frac{\dd t}{\sqrt{q(r,t)}}
 \leq\sqrt2\pi r.
\end{align*}
The equatorial curve $\rho\mapsto e_\rho$, $2\leq\rho\leq r$,
has length $\sqrt{129/8}(r-2)$.
Together with a path from $p_0$ to $e_2$ in $D_0$, these curves
stay in $\Omega_r$ and yield
\begin{align}
 d_{\Omega_r}(p_0,y)
 <C_0+\sqrt{129/8}(r-2)+\sqrt2\pi r
 <C_0+16r.
 \label{eq:fixed3-distance-upper}
\end{align}
Consequently, $d_{\Omega_R}(p_0,y)<C_0+16R$ for every
$y\in\Omega_R$, including $y\in D_0$.
For $T=s/32$ under the stated assumptions, $T\geq4$ and
$C_0+d_g(p,p_0)+16T\leq3s/4$. Thus every $y\in\Omega_T$
satisfies $d_g(p,y)<3s/4$, which proves the claimed inclusion.
\end{proof}

\begin{theorem}\label{thm:fixed3-main}
There exists a smooth complete metric $g$ on $\R^3$ with
$\Ric_g>0$ such that
\begin{align}
 \lim_{s\to\infty}\frac1s\int_{B_p^g(s)}\Sc_g\dd\mu_g=+\infty
 \qquad \forall p\in\R^3.
 \label{eq:fixed3-main-limits}
\end{align}
\end{theorem}

\begin{proof}
Fix the metric supplied by Proposition~\ref{prop:fixed3-extension}.

Proposition~\ref{prop:log-positive} and \eqref{eq:lambda-choice-growth} with $\lambda=1/4$ give
\begin{align}
 \int_{\Sigma_r}N\Sc_g\dd A_r\ge\frac\pi2 r^{1/32}
 \qquad(r\ge2).
 \label{eq:fixed3-scalar-radial-density}
\end{align}

Set $T=s/32$ and take $s$ as in Lemma~\ref{lem-domain-with-respet-to-ball}. Positivity of scalar curvature and \eqref{eq:fixed3-scalar-radial-density} give
\begin{align}
 \int_{B_p^g(s)}\Sc_g\dd\mu_g
 &\ge\frac\pi2\int_{T/2}^T r^{1/32}\dd r
 \ge\frac{\pi T}{4}(T/2)^{1/32}.
\end{align}

Substituting $T=s/32$ proves
\begin{align}
 \frac1s\int_{B_p^g(s)}\Sc_g\dd\mu_g
 \ge\frac\pi{128}\left(\frac{s}{64}\right)^{1/32}.
 \nonumber 
\end{align}
The conclusion follows from the above.
\end{proof}

\begin{corollary}\label{cor:fixed3-local-rescaling}
Fix any $p\in\R^3$ and any sequence $s_j>0$ with $s_j\to\infty$.
The metrics $g_j^{\mathrm{nc}}=s_j^{-2}g$ on $\R^3$ are smooth
and complete, satisfy $\Ric_{g_j^{\mathrm{nc}}}>0$, and obey
\begin{align}
 \lim_{j\to\infty}\int_{B_p^{g_j^{\mathrm{nc}}}(1)}
       \Sc_{g_j^{\mathrm{nc}}}\dd\mu_{g_j^{\mathrm{nc}}}=+\infty.
 \label{eq:local-rescaling-conclusions}
\end{align}
\end{corollary}
\begin{proof}
Constant scaling preserves smoothness, completeness, and the sign
of the Ricci tensor, then
\begin{align}
 B_p^{g_j^{\mathrm{nc}}}(1)&=B_p^g(s_j),\nonumber\\
 \int_{B_p^{g_j^{\mathrm{nc}}}(1)}
       \Sc_{g_j^{\mathrm{nc}}}\dd\mu_{g_j^{\mathrm{nc}}}
 &=s_j^{-1}\int_{B_p^g(s_j)}\Sc_g\dd\mu_g.
 \label{eq:local-rescaling-identities}
\end{align}
Theorem~\ref{thm:fixed3-main} proves the limit. 
\end{proof}

\section{Compact truncations and the compact local counterexamples}
\label{sec:truncations}\label{sec:compact-consequences}
Fix $g,D_0,p_0,C_0$ as in the proof of
Theorem~\ref{thm:fixed3-main}. For $T>2$, define
\begin{align}
 \Omega_T=D_0\cup([2,T]\times\Sph^2),\qquad
 \Sigma_T=\partial\Omega_T.
 \label{eq:compact-fixed-domains}
\end{align}

\begin{proposition}\label{prop:simple-truncations}
Each $\Omega_T$ is diffeomorphic to the closed three-ball, and
its double by the identity boundary map is diffeomorphic to
$\Sph^3$. With the inward normal convention,
\begin{align}
 \mathrm{II}_{\Sigma_T}
 \geq\frac{7}{8T\max_{\Sigma_T}N}\,h_T>0.
 \label{eq:compact-fixed-boundary}
\end{align}
If $T\geq\max\{4,C_0/8\}$, then
\begin{align}
 d_{\Omega_T}(p_0,x)<24T\qquad(x\in\Omega_T).
 \label{eq:compact-fixed-intrinsic-distance}
\end{align}
\end{proposition}

\begin{proof}
The inward unit normal to $\Omega_T$ is $-N^{-1}\partial_r$;
therefore $\mathrm{II}=(2N)^{-1}\partial_rh_r$. 

Its two principal curvatures are
$\left.a_r/(Na)\right|_{r=T}$ and $7/(8TN)$.
By \eqref{eq:direct-a-derivatives}, both are at least
$7/(8TN)$ away from the poles and equator.
Smoothness extends this bound over $\Sigma_T$, proving
\eqref{eq:compact-fixed-boundary}. 

For fixed $T$, the positive
function $N$ has a finite maximum on $\Sigma_T$; a bound uniform
in $T$ is not needed. The curves used in
\eqref{eq:fixed3-distance-upper} stay inside the truncation at
their terminal radius. Their lengths are less than
$C_0+16T\leq24T$, proving the distance bound, including points
of $D_0$.

The region $\Omega_T$ is obtained from the closed three-ball $D$ by attaching a finite outward collar. Reparametrizing this collar gives a diffeomorphism $\Omega_T\cong D$ compatible with boundary collars; applying it to both copies identifies the double with the double of $D$, hence with $\Sph^3$.
\end{proof}

\begin{theorem}\label{thm:main-unboundedness}
There exists a sequence of smooth metrics $g_j^{\mathrm c}$ on
$\Sph^3$ such that
\begin{align}
 \Ric_{g_j^{\mathrm c}}>0,\qquad
 \diam(\Sph^3,g_j^{\mathrm c})<\tfrac12,\qquad
 \lim_{j\to\infty}\int_{\Sph^3}
       \Sc_{g_j^{\mathrm c}}\dd\mu_{g_j^{\mathrm c}}=+\infty.
 \label{eq:compact-main-conclusion}
\end{align}
\end{theorem}

\begin{proof}
Let $M_T$ be the double of two copies $\Omega_T^+$ and
$\Omega_T^-$, and write $g_T^0$ for their continuous glued
metric. For each fixed $T>8$, both pieces have positive Ricci
curvature, their compact boundaries are isometric, and
Proposition~\ref{prop:simple-truncations} gives a strictly
positive sum of their inward boundary forms. Apply
Theorem~\ref{thm:unified-positive-Ricci-gluing}, with smoothing
supported in the two copies of $\{3T/4<r\leq T\}$. Choose its
$C^0$ approximation sufficiently close that the resulting smooth
metric satisfies
\begin{align}
 \Ric_{\widetilde g_T}&>0,\qquad
 \tfrac12g_T^0\leq\widetilde g_T\leq2g_T^0,\nonumber\\
 \widetilde g_T&=g\quad\text{on both copies of }\Omega_{T/2}.
 \label{eq:compact-fixed-smoothed}
\end{align}

Assume also $T\geq C_0/8$. Any two points in one copy can be
joined through $p_0$ by curves of total length less than $48T$.
For points in different copies, fix a point of $\Sigma_T$ and
join each point to it in its own copy by curves of length less
than $48T$. These curves give
\begin{align}
 \diam(M_T,\widetilde g_T)\leq96\sqrt2\,T,
 \label{eq:compact-fixed-diameter}
\end{align}
using the upper tensor bound in
\eqref{eq:compact-fixed-smoothed}.

Both copies of $\{T/4\leq r\leq T/2\}$ retain the original
metric. The positivity of scalar curvature on their complement
and \eqref{eq:fixed3-scalar-radial-density} yield
\begin{align}
 \int_{M_T}\Sc_{\widetilde g_T}\dd\mu_{\widetilde g_T}
 &\geq\pi\int_{T/4}^{T/2}r^{1/32}\dd r
 \geq\frac{\pi T}{4}(T/4)^{1/32}.
 \label{eq:compact-fixed-scalar}
\end{align}
This lower bound is taken entirely in unchanged regions; it
requires no curvature convergence under $C^0$ approximation.
Define
\begin{align}
 G_T=(512T)^{-2}\widetilde g_T.
 \label{eq:simple-compact-normalization}
\end{align}
The scaling identities give
\begin{align}
 \Ric_{G_T}&>0,\qquad
 \diam(M_T,G_T)\leq\frac{3\sqrt2}{16}<\tfrac12,\nonumber\\
 \int_{M_T}\Sc_{G_T}\dd\mu_{G_T}
 &\geq\frac\pi{2048}(T/4)^{1/32}.
 \label{eq:compact-fixed-estimates}
\end{align}
Choose any sequence $T_j\to\infty$ with
$T_j\geq\max\{16,C_0/8\}$, and pull $G_{T_j}$ back to
$\Sph^3$ using the diffeomorphisms already proved. These
pullbacks define $g_j^{\mathrm c}$. They are complete by
compactness, their diameter is less than $1/2$, and the displayed
scalar-curvature integrals tend to infinity. Thus their open
radius-one balls equal the whole sphere for every center,
proving all the assertions.
\end{proof}

\begin{proposition}\label{prop:volume-collapse}
The fixed metric of Theorem~\ref{thm:fixed3-main} satisfies
\begin{align}
 \lim_{s\to\infty}s^{-3}\Vol_gB_p^g(s)=0
 \qquad\text{for every }p\in\R^3.
 \label{eq:volume-global-limit}
\end{align}
For the metrics in Corollary~\ref{cor:fixed3-local-rescaling} and
Theorem~\ref{thm:main-unboundedness}, respectively,
\begin{align}
 \lim_{j\to\infty}
 \Vol_{g_j^{\mathrm{nc}}}B_p^{g_j^{\mathrm{nc}}}(1)=0,\qquad
 \lim_{j\to\infty}\Vol_{g_j^{\mathrm c}}(\Sph^3)=0.
 \label{eq:volume-local-limits}
\end{align}
\end{proposition}

\begin{proof}
Since $(1+q)^{1/4}\leq1+q$, equation~\eqref{eq:qbound} gives
\begin{align*}
 \frac{N^2}{q}
 \leq\frac{129}{8q}+\frac18
 \leq\frac{64}{1-t^2}
 \qquad(0<t<1).
\end{align*}
Using $\dd A_r=r^{15/8}tq^{-1/2}\dd t\dd\phi$ on each
hemisphere and $\int_0^1t(1-t^2)^{-1/2}\dd t=1$, we obtain
\begin{align}
 \int_{\Sigma_r}N\dd A_r\leq32\pi r^{15/8}.
 \label{eq:log-volume-density}
\end{align}
For fixed $p$, set $d_0=d_g(p,p_0)$.
The exterior bound $g\geq16\dd r^2$ implies
$4(r(y)-2)\leq d_g(p_0,y)$ for $y\notin D_0$.
Hence $B_p^g(s)\subset D_0\cup\{2\leq r<2+(s+d_0)/4\}$,
and integration of \eqref{eq:log-volume-density} gives
\begin{align}
 \Vol_g B_p^g(s)
 \leq\Vol_g D_0+\frac{256\pi}{23}
       \bigl(2+(s+d_0)/4\bigr)^{23/8}.
 \label{eq:fixed3-ball-volume-bound}
\end{align}
Division by $s^3$ proves \eqref{eq:volume-global-limit}.
The identity
\begin{align*}
 \Vol_{g_j^{\mathrm{nc}}}B_p^{g_j^{\mathrm{nc}}}(1)
 =s_j^{-3}\Vol_g B_p^g(s_j)
\end{align*}
then proves the noncompact limit in \eqref{eq:volume-local-limits}.

For $T\geq2$, the same density estimate gives
\begin{align*}
 \Vol_g\Omega_T
 \leq\Vol_g D_0+\frac{256\pi}{23}T^{23/8}.
\end{align*}
There are two copies of $\Omega_T$, and
$\widetilde g_T\leq2g_T^0$ multiplies their volume densities by
at most $2^{3/2}$. Therefore
\eqref{eq:simple-compact-normalization} gives
\begin{align}
 \Vol_{G_T}M_T
 \leq\frac{2^{5/2}}{512^3}
       \left(\Vol_gD_0+\frac{256\pi}{23}\right)T^{-1/8}.
 \label{eq:simple-compact-volume-final}
\end{align}
This proves the compact limit in \eqref{eq:volume-local-limits}.
\end{proof}

\section{Appendix}\label{Appendix-curvature}

\begin{proposition}\label{prop:two-variable-Ricci}
Let $U\subset\R^2$ be open, with coordinates $(r,t)$, and let
$N,a,b:U\to(0,\infty)$ be smooth. On $U\times\Sph^1$, where
$\phi$ has period $2\pi$, define
\begin{align}
 g=N^2\dd r^2+a^2\dd t^2+b^2\dd\phi^2.
 \label{eq:two-variable-metric}
\end{align}
Put
\begin{align*}
 e_0=N^{-1}\partial_r,\qquad
 e_1=a^{-1}\partial_t,\qquad
 e_2=b^{-1}\partial_\phi,
\end{align*}
and write $R_{ij}=\Ric_g(e_i,e_j)$. Then
\begin{align}
 R_{00}
 &=-\frac1{Na}\left\{
       \partial_r\left(\frac{a_r}{N}\right)
       +\partial_t\left(\frac{N_t}{a}\right)\right\}
   -\frac1{bN^2}\left(
       b_{rr}-\frac{N_r}{N}b_r+\frac{NN_t}{a^2}b_t
     \right),
 \label{eq:two-variable-R00}\\
 R_{11}
 &=-\frac1{Na}\left\{
       \partial_r\left(\frac{a_r}{N}\right)
       +\partial_t\left(\frac{N_t}{a}\right)\right\}
   -\frac1{ba^2}\left(
       b_{tt}+\frac{aa_r}{N^2}b_r-\frac{a_t}{a}b_t
     \right),
 \label{eq:two-variable-R11}\\
 R_{01}
 &=-\frac1{bNa}\left(
       b_{rt}-\frac{N_t}{N}b_r-\frac{a_r}{a}b_t
     \right),
 \label{eq:two-variable-R01}\\
 R_{22}
 &=-\frac1{bN^2}\left(
       b_{rr}-\frac{N_r}{N}b_r+\frac{NN_t}{a^2}b_t
     \right)
   -\frac1{ba^2}\left(
       b_{tt}+\frac{aa_r}{N^2}b_r-\frac{a_t}{a}b_t
     \right),
 \label{eq:two-variable-R22}\\
 R_{02}&=R_{12}=0.
 \label{eq:two-variable-R02-R12}
\end{align}
\end{proposition}

\begin{proof}
Since the metric is diagonal and independent of $\phi$, the nonzero
Christoffel symbols, up to symmetry in the lower indices, are
\begin{align}
 &\Gamma^r_{rr}=\frac{N_r}{N},\qquad
 \Gamma^r_{rt}=\frac{N_t}{N},\qquad
 \Gamma^r_{tt}=-\frac{aa_r}{N^2},\qquad
 \Gamma^r_{\phi\phi}=-\frac{bb_r}{N^2},\nonumber\\
 &\Gamma^t_{rr}=-\frac{NN_t}{a^2},\qquad
 \Gamma^t_{rt}=\frac{a_r}{a},\qquad
 \Gamma^t_{tt}=\frac{a_t}{a},\qquad
 \Gamma^t_{\phi\phi}=-\frac{bb_t}{a^2},\nonumber\\
 &\Gamma^\phi_{r\phi}=\frac{b_r}{b},\qquad
 \Gamma^\phi_{t\phi}=\frac{b_t}{b}.
 \label{eq:two-variable-Christoffel}
\end{align}

Substitute these Christoffel symbols into
\begin{align*}
 (\Ric_g)_{ij}
 =\partial_k\Gamma^k_{ij}-\partial_j\Gamma^k_{ik}
  +\Gamma^k_{ij}\Gamma^\ell_{k\ell}
  -\Gamma^\ell_{ik}\Gamma^k_{j\ell},
\end{align*}
where repeated indices are summed over $r,t,\phi$.
Dividing the $rr$, $tt$, $rt$, and $\phi\phi$ components by
$N^2$, $a^2$, $Na$, and $b^2$, respectively, yields
\eqref{eq:two-variable-R00}--\eqref{eq:two-variable-R22}.
The $r\phi$ and $t\phi$ components vanish, giving
\eqref{eq:two-variable-R02-R12}.
\end{proof}

\begin{lemma}\label{lem:identical-hemisphere-smoothness}
Let $N,a,b$ be smooth positive functions on
$[1,\infty)\times(0,1)$, jointly smooth up to $r=1$.
Parametrize the two hemispheres of $\Sph^2$ by
\begin{align*}
 (y_1,y_2,y_3)=(\pm\sqrt{1-t^2},t\cos\phi,t\sin\phi),
 \qquad \phi\in\R/(2\pi\mathbb Z),
\end{align*}
and use the same functions $N,a,b$ on both copies of $0<t<1$.
Then
\begin{align*}
 g=N(r,t)^2\dd r^2+a(r,t)^2\dd t^2+b(r,t)^2\dd\phi^2
\end{align*}
extends to a smooth positive definite metric on
$[1,\infty)\times\Sph^2$ if the following endpoint
conditions hold. For every finite $R>1$, there are smooth positive
functions $N_0,a_0,b_0,N_1,a_1,b_1$ of $(r,z)$ on a neighborhood
of $[1,R]\times\{0\}$ in $[1,\infty)\times[0,\infty)$ such that,
for $1\le r\le R$, near $t=0$,
\begin{align}
 N(r,t)&=N_0(r,t^2),\qquad a(r,t)=a_0(r,t^2),\nonumber\\
 b(r,t)&=t\,b_0(r,t^2),\qquad a_0(r,0)=b_0(r,0)>0,
 \label{eq:identical-hemisphere-poles}
\end{align}
and, near $t=1$,
\begin{align}
 N(r,t)&=N_1(r,1-t^2),\qquad b(r,t)=b_1(r,1-t^2),\nonumber\\
 \frac{\sqrt{1-t^2}}{t}\,a(r,t)&=a_1(r,1-t^2).
 \label{eq:identical-hemisphere-equator}
\end{align}
All endpoint neighborhoods may depend on $R$, and all smoothness
conditions are joint in $r,z$, including their boundary values.
No equality between $a_1(r,0)$ and $b_1(r,0)$ is required.
\end{lemma}

\begin{proof}
Near either pole use the smooth
coordinates $X=t\cos\phi$, $Y=t\sin\phi$, and put
$z=X^2+Y^2=t^2$. On the punctured coordinate disk,
\begin{align*}
 \dd X^2+\dd Y^2&=\dd t^2+t^2\dd\phi^2,\qquad
 X\dd X+Y\dd Y=t\dd t.
\end{align*}
It follows from \eqref{eq:identical-hemisphere-poles} that
\begin{align}
 g={}&N_0(r,z)^2\dd r^2+b_0(r,z)^2(\dd X^2+\dd Y^2)\nonumber\\
 &+\frac{a_0(r,z)^2-b_0(r,z)^2}{z}
                         (X\dd X+Y\dd Y)^2.
 \label{eq:identical-hemisphere-Cartesian}
\end{align}
The numerator in the quotient vanishes at $z=0$, and the quotient
has the smooth extension
\begin{align*}
 \int_0^1\partial_z(a_0^2-b_0^2)(r,\lambda z)\dd\lambda.
\end{align*}
Thus every coefficient in
\eqref{eq:identical-hemisphere-Cartesian} is smooth. At the pole
the last term vanishes, and the metric is
\begin{align*}
 N_0(r,0)^2\dd r^2+b_0(r,0)^2(\dd X^2+\dd Y^2),
\end{align*}
which is positive definite. This argument applies to both poles
because the functions used on the two hemispheres are identical.

At the equator take the signed coordinate $v=y_1$, so that
$t=\sqrt{1-v^2}$ and
\begin{align*}
 \dd t=-\frac{v}{\sqrt{1-v^2}}\dd v.
\end{align*}
Here $(r,v,\phi)$ is a smooth coordinate system after restricting
$\phi$ to any circle coordinate chart. Substitution of
\eqref{eq:identical-hemisphere-equator} gives, for $v\ne0$,
\begin{align}
 g=N_1(r,v^2)^2\dd r^2+a_1(r,v^2)^2\dd v^2
                         +b_1(r,v^2)^2\dd\phi^2.
 \label{eq:identical-hemisphere-signed}
\end{align}
The same formula holds on both sides of the equator, and its
coefficients extend smoothly and positively to $v=0$. Together
with positivity on $0<t<1$ and the polar calculation, this proves
sufficiency. The extensions agree on overlaps since they agree on
the dense regular coordinate region. In particular the extension
is unique. All calculations are joint with $r$ and require no
condition at infinity.
\end{proof}

\begin{theorem}\label{thm:unified-positive-Ricci-gluing}
Let $(M_1,g_1)$ and $(M_2,g_2)$ be smooth three-manifolds with compact
nonempty isometric boundaries.  Fix a boundary isometry and use it to identify
all boundary tensors in the inequality below.  Assume $\Ric_{g_i}>0$.
Use inward unit normals $\vec{n}$ and the convention \eqref{eq:boundary-sign-definition}.
If
\begin{align}
 \mathrm{II}_{\partial M_1}+\mathrm{II}_{\partial M_2}>0,
 \label{eq:unified-gluing-II-condition}
\end{align}
then the continuous metric obtained by identifying the boundaries can be
replaced by a smooth metric of positive Ricci curvature.  The replacement
can be confined to an arbitrarily small neighborhood of the gluing
hypersurface, leaves the original metrics unchanged outside that
neighborhood, and can be chosen arbitrarily close in $C^0$ to the continuous
glued metric.
\end{theorem}

\noindent
This is Perelman's gluing theorem; see
\cite[Theorem~2 and its proof in Section~2.2]
{BotvinnikWalshWraith2019}.
The arbitrary localization and $C^0$-approximation follow from
equation~(2-1), Proposition~8, and the final smoothing on p.~3020
of that paper.

\end{document}